\documentclass[12pt, twoside, leqno]{article}

\usepackage{amsmath,amsthm}
\usepackage{amssymb}

\usepackage{enumerate}

\usepackage{graphicx}

\newtheorem{thm}{Theorem}[section]
\newtheorem{cor}[thm]{Corollary}
\newtheorem{lem}[thm]{Lemma}
\newtheorem{prop}[thm]{Proposition}

\theoremstyle{definition}
\newtheorem{defin}[thm]{Definition}
\newtheorem{rem}[thm]{Remark}
\newtheorem{exa}[thm]{Example}

\numberwithin{equation}{section}

\begin{document}


\baselineskip=17pt


\title{Topological transitivity and wandering intervals for group actions on the line $\mathbb R$}

\author{Enhui Shi\\
School of Mathematical Sciences\\
Soochow University\\
Suzhou 215006, P. R. China\\
E-mail:  ehshi@suda.edu.cn \\
\\
Lizhen Zhou\\
School of Mathematical Sciences\\
Soochow University\\
Suzhou 215006, P. R. China\\
E-mail: zhoulizhen@suda.edu.cn}

\date{}

\maketitle


\renewcommand{\thefootnote}{}



\renewcommand{\thefootnote}{\arabic{footnote}}
\setcounter{footnote}{0}


\begin{abstract}
For every group $G$, we show that either $G$ has a topologically transitive action on the line $\mathbb R$ by orientation-preserving homeomorphisms,
or every orientation-preserving action of $G$ on $\mathbb R$ has a wandering interval. According to this result, all groups are divided into two types:
transitive type and wandering type, and the types of several groups are determined. We also show that every finitely generated orderable group of wandering type is indicable.
As a corollary, we show that if a higher rank lattice $\Gamma$ is orderable, then $\Gamma$ is of transitive type. 
\end{abstract}

\section{Preliminaries}

 Let $X$ be a topological space and let ${\rm Homeo}(X)$ be
the homeomorphism group of $X$. Suppose $G$ is a group. A group
homomorphism $\phi: G\rightarrow {\rm Homeo}(X)$ is called an {\it
action} of $G$ on $X$; the action is said to be {\it faithful} if
$\phi$ is injective. If $G$ is a subgroup of ${\rm Homeo}(X)$, then
the action of $G$ on $X$ always refers to the inclusion homomorphism $\iota: G\hookrightarrow {\rm Homeo}(X)$.
For brevity, we usually use $gx$ or $g(x)$
instead of $\phi(g)(x)$ for $g\in G$ and $x\in X$. The {\it orbit}
of $x\in X$ under the action of $G$ is the set $Gx\equiv\{gx:g\in
G\}$;  $x$ is called a {\it fixed point} of $\phi$ or of $G$, if $gx=x$ for all $g\in G$.
We use ${\rm Fix}(G)$ to denote the fixed point set of $G$; use ${\rm Fix}(g)$ to denote
the fixed point set of the cyclic group $\langle g\rangle$ generated by $g\in G$. A subset $Y$ of $X$ is called {\it
$G$-invariant}, if $g(Y)\subset Y$ for all $g\in G$.

Let $\phi$ be an action of group $G$ on a topological space $X$. The action $\phi$ (or $G$) is said to be
{\it topologically transitive}, if for every nonempty open subsets $U$ and $V$ of $X$, there
is some $g\in G$ such that $g(U)\cap V\not=\emptyset$. It is well known that, when $G$ is countable and $X$ is a Polish space, $G$ is
topologically transitive if and only if there is a point $x\in X$ such that the orbit $Gx$ is dense
in $X$. Furthermore, $G$ is said to be {\it minimal} if for every $x\in X$ the orbit $Gx$ is dense in $X$;
this is equivalent to saying that there is no proper $G$-invariant nonempty closed subset of $X$.
A closed subset $Y$ of $X$ is said to be {\it minimal}, if $Y$ is $G$-invariant and the restriction action $G|_Y$ of $G$ to $Y$
is minimal. An argument using Zorn's lemma shows the existence of minimal sets when $X$ is a compact metric space,
but this is not true in general when $X$ is not compact.

Let $\mathbb R$ be the real line and let ${\rm Homeo}_+(\mathbb R)$ be the orientation-preserving
homeomorphism group of $\mathbb R$. A group homomorphism $\phi$ from $G$ to ${\rm Homeo}_+(\mathbb R)$
is called an {\it orientation-preserving action} of $G$ on $\mathbb R$. An open interval $(a,b)\subset \mathbb R$ is said to be
a {\it wandering interval} of $\phi$ or of $G$ if, for every $g\in G$, either the restriction $g|_{(a,b)}={\rm Id}_{(a,b)}$
or $g((a, b))\cap (a, b)=\emptyset$.

For any $\alpha\in\mathbb R$, define $L_\alpha, M_\alpha:\mathbb R\rightarrow \mathbb R$ by letting $L_\alpha(x)=x+\alpha$
and $M_\alpha(x)=\alpha x$ for every $x\in \mathbb R$. We use $\mathbb Z^n$ to denote the free abelian group of rank $n$. Now we give some examples to
illustrate the above notions, which will be used in section $4$.

\begin{exa}
Every open interval $(a, b)$ with $b-a<1$ is a wandering interval for the $\mathbb Z$ action generated by $L_1$.
\end{exa}

\begin{exa}
Let $\alpha$ be an irrational number, then the $\mathbb Z^2$ action generated by $L_1$ and $L_\alpha$ on $\mathbb R$
is minimal.
\end{exa}

\begin{exa}
Let $n$ be a positive integer. Let $T=L_1$ and let $S=M_n$. Then
$S^{-m}TS^m(x)=x+n^{-m}$ for all non negative integers $m$ and all $x\in \mathbb R$, which clearly implies the minimality of the
action of the group $G$ generated by $S$ and $T$.
\end{exa}

\begin{exa}
Let $f=L_1$ and let $k\geq 2$ be a positive integer. Define a homeomorphism $g$ on $\mathbb R$ by setting
$$g(x)=(x-n)^{2^{(-1)^nk^{-n}}}+n$$ for all integers $n$ and all $x\in [n, n+1)$. Then, for $x\in [n, n+1)$, we have
$$\begin{array}{rl}
    fgf^{-1}(x)&=(x-1-(n-1))^{2^{(-1)^{n-1}k^{-(n-1)}}}+(n-1)+1\\
               &=(x-n)^{2^{(-1)^{n-1}k^{-(n-1)}}}+n,
\end{array}
$$
and
$$\begin{array}{rl}
    g^{-k}(x)&=(x-n)^{2^{(-1)^{n+1}k^{-n}k}}+n=(x-n)^{2^{(-1)^{n-1}k^{-(n-1)}}}+n.
\end{array}
$$
So $fgf^{-1}=g^{-k}$. Since $f^mg^lf^n(\frac{1}{2})=(\frac{1}{2})^{2^{(-1)^nk^{-n}l}}+n+m$ for all integers $m, l, n$, the set
$\{f^mg^lf^n(\frac{1}{2}):m,l,n\in\mathbb Z\}$ is dense in $\mathbb R$. This implies that the action by the group $\langle f, g\rangle$ generated by
$f$ and $g$ is topologically transitive.
\end{exa}

\section{Background and main results}

The dynamical system for group actions on one-manifolds has been intensively studied. One may consult \cite{Gh, Na1} for a
systematic introduction to this area. Group actions on the real line $\mathbb R$ are closely related to the study of orderability of a group.
In fact, a countable group $G$ is orderable (that is, $G$ admits a left-invariant total order relation) if and only if
it admits a faithful orientation-preserving action on the real line (see {\cite[Prop. 2.1]{Na2}}). Many important groups coming from
geometry and topology are known to be orderable or nonorderable (see e.g. \cite{BRW, DDRW, SW, Wi2}). In addition, an orderable group
may possess some interesting algebraic properties (see e.g. \cite{Na0, Rh,  Wi1}).

The purpose of the paper is to classify group actions on $\mathbb R$ by means of topological transitivity. One may consult \cite{SZ} for
some related investigations. We first give the following dichotomy theorem.

\begin{thm}
Let $G$ be a group. Then either $G$ has a topologically transitive action on the line $\mathbb R$ by orientation-preserving homeomorphisms,
or every orientation-preserving action of $G$ on $\mathbb R$ has a wandering interval.
\end{thm}

We should note that the ``dichotomy phenomenon" in Theorem 2.1 is far from being true for group actions on spaces of dimension $\geq 2$.
For example, if $D$ is the closed unit disk in the plane and $S^2$ is the unit sphere in $\mathbb R^3$, then any one point union of $D$ and $S^2$
admits no topologically transitive homeomorphism but admits a homeomorphism with no wandering open set.

Theorem 2.1  motivates us to give the following definition.

\begin{defin}
A group $G$ is of {\it transitive type} if it has a topologically transitive action on the line $\mathbb R$ by
orientation-preserving homeomorphisms; it is of {\it wandering type} if every orientation-preserving action of $G$ on $\mathbb R$
has a wandering interval.
\end{defin}

Recall that a group $G$ is {\it poly-cyclic} (resp. {\it super-poly-cyclic}) if it admits a decreasing sequence of subgroups
$G=N_0\rhd N_1\rhd...\rhd N_k=\{e\}$ for some positive integer $k$ such that $N_{i+1}$ is normal in $N_i$ (resp. $N_{i+1}$ is normal in $G$) and $N_i/N_{i+1}$ is cyclic for each $i\leq k-1$; it is called {\it poly-infinite-cyclic} (resp. {\it super-poly-infinite-cyclic}) if $N_i/N_{i+1}$ is infinitely cyclic for each $i\leq k-1$. It is well known that all poly-cyclic groups are solvable and all finitely generated torsion free nilpotent groups are super-poly-infinite-cyclic.

Suppose $G=N_0\rhd N_1\rhd...\rhd N_k=\{e\}$ is super-poly-infinite-cyclic. Take $f_i\in N_i\setminus N_{i+1}$ such that $N_i/N_{i+1}=\langle f_iN_{i+1}\rangle$ for each $i$. Then $f_if_{i+1}N_{i+2}f_i^{-1}=f_{i+1}^{n_i}N_{i+2}$ where $n_i=\pm 1$. We call the $(k-1)$-tuple $(n_0, n_1, ..., n_{k-2})$ {\it the name of $G$}. Clearly, the name of $G$ is independent of the choice of $f_i$.

For each integer $n$, the {\it Baumslag-Solitar group} $B(1, n)$ is the group $\langle a, b:  ba=a^nb \rangle$; $B(1, -1)$ is the fundamental group of the  Klein Bottle, which is a classical example
being of orderable but not bi-orderable (see \cite[Exercise 2.2.68]{Na1}).

\begin{thm}
The following groups are of transitive type: the nonabelian free group $\mathbb Z*\mathbb Z$; any super-poly-infinite-cyclic group $G=N_0\rhd N_1\rhd...\rhd N_k=\{e\}$ having the name $(n_0, n_1, ..., n_{k-2})$ with some $n_i=1$; any poly-infinite-cyclic, non super-poly-infinite-cyclic group $G$; the Baumslag-Solitar group $B(1, n)$ with $n\not=0$ and $n\not=-1$.

The following groups are of wandering type: finite groups;  the infinite cyclic group $\mathbb Z$; $SL(2, \mathbb Z)$; finite index subgroups of $SL(n,\mathbb Z)$ with $n\geq 3$; any super-poly-infinite-cyclic group $G=N_0\rhd N_1\rhd...\rhd N_k=\{e\}$ having the name $(-1, -1, ..., -1)$; the Baumslag-Solitar group $B(1, -1)$.
\end{thm}

Recall that a group is {\it indicable} if it has a homomorphism onto the infinite cyclic group.
One may consult \cite{Be, LMR, Wi1} for the discussions about indicability of orderable groups.

\begin{thm}
If $G$ is a finitely generated nontrivial orderable group of wandering type, then $G$ is indicable.
\end{thm}

{\it A higher rank lattice} is a lattice of a simple Lie group with finite center and with real rank $\geq 2$.
The 1-dimensional Zimmer's rigidity conjecture says that every continuous action of a higher rank lattice on the
circle $\mathbb S^1$ must factor through a finite group action. Though the conjecture is still open now,
Ghys (see \cite{Ghy}) and Burger-Monod (see \cite{Bu}) proved independently the existence of periodic points for such actions.
This implies that 1-dimensional Zimmer's rigidity conjecture is equivalent to that no higher rank lattice is orderable.
We get immediately the following corollary by Theorem 2.4 and Theorem 6.5 in the appendix.
\begin{cor}
Suppose $G$ is a higher rank lattice. If $G$ is orderable,
then it is of transitive type.
\end{cor}

\section{The dichotomy theorem}

\begin{lem}
Let $G$ be a group. Suppose $G$ has no topologically transitive action on the line $\mathbb R$ by orientation-preserving homeomorphisms.
Then, for every action $\phi: G\rightarrow {\rm Homeo}_+(\mathbb R)$ and for every $x\in \mathbb R$, $\overline {Gx}$ is countable; in particular,
 $\overline {Gx}$ is nowhere dense.
\end{lem}
\begin{proof} Assume to the contrary that there is some action $\phi_0: G\rightarrow {\rm Homeo}_+(\mathbb R)$ and some $x_0\in \mathbb R$ such
 that $\overline {Gx_0}$ is uncountable. Then, by collapsing the maximal open intervals of $\mathbb R\setminus \overline {Gx_0}$, we get an
 induced topologically transitive action of $G$ on either $[0, 1]$, or $[0, 1)$, or $(0, 1]$, or $(0, 1)$. By removing the endpoints of the phase
 space of the induced action if necessary, we get a topologically transitive action of $G$ on $\mathbb R$ by orientation-preserving homeomorphisms.
 This contradicts the assumption.

\end{proof}

{\noindent \it Proof of Theorem 2.1.} Assume to the contrary that the following two items hold simultaneously:

{\noindent\bf (a)} $G$ has no topologically transitive action on $\mathbb R$ by orientation-preserving homeomorphisms;

{\noindent\bf (b)} there is an action $\phi:G\rightarrow  {\rm Homeo}_+(\mathbb R)$ such that $\phi$
has no wandering interval.

From Assumption (b) and the definition of wandering interval, there is some $x_1\in \mathbb R$ and some $g_1\in G$
such that $x_1<\phi(g_1)(x_1)$ (otherwise, $\phi(g)={\rm Id}_{\mathbb R}$ for all $g\in G$; then every open interval in $\mathbb R$ is wandering).
Without loss of generality, we suppose that $\{x_1, \phi(g_1)(x_1)\}\subset (0, 1)$.
Set $U_0=(0, 1)$. For the simplicity of notations, we use $g(x)$ instead of $\phi(g)(x)$ in what follows.

Now we define inductively a sequence of open intervals $U_i$ and $g_i\in G$, $i=1, 2, ...$, such that

{\noindent\bf(1)} for each $i\geq 1$, $U_{i-1}\supset U_{i}$;

{\noindent\bf(2)} for every $g\in G$ and every $i\geq 1$, either $g(U_i)=U_i$, or $g(U_i)\cap U_i=\emptyset$;

{\noindent\bf(3)} for every $g\in G$ and every $i\geq 1$, ${\rm diam}(g(U_i)\cap[0, 1])<{1\over i}$;

{\noindent\bf(4)} for each $i\geq 1$, ${\overline {U_i}}\cap g_i({\overline {U_i}})=\emptyset$ and
${\overline {U_i}}\cup g_i({\overline {U_i}})\subset U_{i-1}$.

 For $i=1$, take a sufficiently small interval $V_1\subset (0, 1)$ such that $x_1\in V_1$,
\begin{equation}
{\overline {V_1}}\cup g_1({\overline {V_1}})\subset U_{0},\ \mbox{and}\ {\overline {V_1}}\cap g_1({\overline {V_1}})=\emptyset.
\end{equation}
Take a sufficiently large positive integer $i_1>1$ such that
\begin{equation}
\frac{1}{i_1}<\frac{1}{2}{\rm diam}(V_1).
\end{equation}
Let $A_1=\{k\frac{1}{i_1}:k=0, 1, ..., i_1\}$ and let $B_1={\overline {GA_1}}$. It follows from
Lemma 3.1 that $B_1$ is a nowhere dense $G$-invariant closed subset of $\mathbb R$. From $(3.2)$ and
the definition of $A_1$, there exists a maximal open interval $U_1$ of $\mathbb R\setminus B_1$ such that
$U_1\subset V_1$. Then $(1)-(4)$ hold for $U_1$ and $g_1$ by $(3.1)$, $(3.2)$ and the definition of $U_1$.

Suppose that for $1\leq i\leq k$ we have defined $U_i$ and $g_i$  which satisfy $(1)-(4)$. Then
define $U_{k+1}$ and $g_{k+1}$ as follows. From Assumption (b), $U_k$ is nonwandering, which together with
$(2)$ implies that there is some point $x_{k+1}\in U_k$ and some $g_{k+1}\in G$ such that $g_{k+1}(x_{k+1})\in U_k$ and $g_{k+1}(x_{k+1})>x_{k+1}$.
Take a sufficiently small open interval $V_{k+1}$ such that $x_{k+1}\in V_{k+1}$,
\begin{equation}
{\overline {V_{k+1}}}\cup g_{k+1}({\overline {V_{k+1}}})\subset U_{k},\ \mbox{and}\ {\overline {V_{k+1}}}\cap g_{k+1}({\overline {V_{k+1}}})=\emptyset.
\end{equation}
Take a sufficiently large positive integer $i_{k+1}>k+1$ such that
\begin{equation}
\frac{1}{i_{k+1}}<\frac{1}{2}{\rm diam}(V_{k+1}).
\end{equation}
Let $A_{k+1}=\{k\frac{1}{i_{k+1}}:k=0, 1, ..., i_{k+1}\}$ and let $B_{k+1}={\overline {GA_{k+1}}}$. Similar to the case of $i=1$,
we get a maximal open interval $U_{k+1}$ of $\mathbb R\setminus B_{k+1}$ which satisfy the conditions $(1)-(4)$.

Now we define a sequence of subsets $G_i$ of $G$ for $i=1, 2, ...$ as follows. Let $G_1=\{e, g_1\}$. Assume $G_i$ have been defined for $1\leq i\leq k$.
Then let $G_{k+1}=G_k\cup \{gg_{k+1}: g\in G_k\}$. For each $k=1, 2,...$, set $\Lambda_k=\cup_{g\in G_k}g(\overline{U_k})$, and set $\Lambda=\cap_{k=1}^\infty \Lambda_k$. It follows from $(3)$ and $(4)$ that $\Lambda$ is homeomorphic to the Cantor set, and, for any point $x\in\Lambda$,
$\overline{Gx}\supset \Lambda$ (one may see  Fig. 1 for the illustration of the ideas of the construction). This implies that $G$ has a topologically transitive orientation-preserving action on $\mathbb R$ by Lemma 3.1, which contradicts the assumption (a).
 \hfill{$\Box$}

\begin{figure}[htbp]
\centering
\includegraphics[scale=0.5]{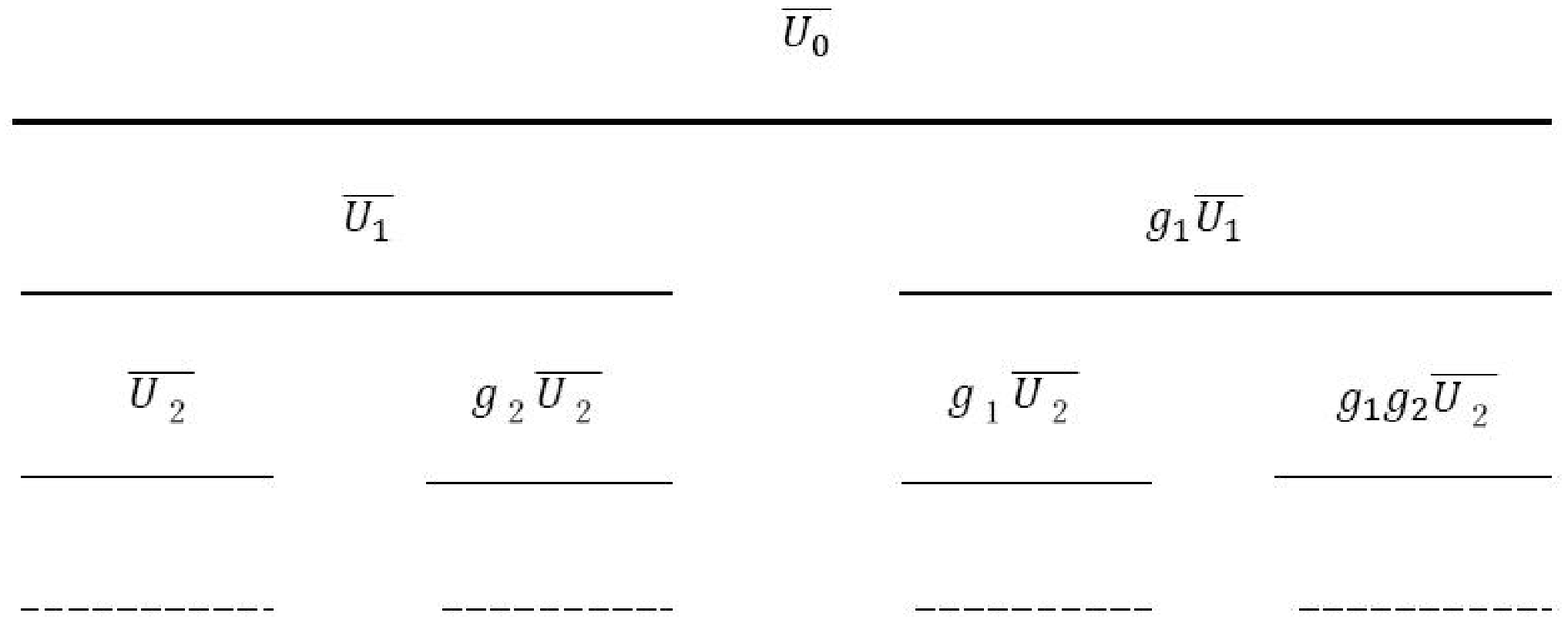}
\centerline{Fig. 1}
\end{figure}

\section{Types of some groups}
In this section, we start to prove Theorem 2.3. From Definition 2.2, we immediately have that all finite groups and the infinite cyclic group
$\mathbb Z$ are of wandering type.

\begin{prop}
Suppose $G=N_0\rhd N_1\rhd...\rhd N_k=\{e\}$ is super-poly-infinite-cyclic and has the name $(-1, -1, ..., -1)$. Then $G$ is of wandering type.
\end{prop}

\begin{proof}
By the hypothesis, we can take a sequence $g_i\in N_i\setminus N_{i+1}$ and take $g_k=e$ such that $N_i=\langle g_i, ..., g_{k-1}\rangle$ and
\begin{equation}
g_ig_{i+1}N_{i+2}=g_{i+1}^{-1}g_iN_{i+2}\ \mbox{ for\ each}\ i: 0\leq i\leq k-2.
\end{equation}
Let $\phi:G\rightarrow {\rm Homeo}_+(\mathbb R)$ be any orientation-preserving action of $G$ on $\mathbb R$ and let $f_i=\phi(g_i)$ for
each $i$. If $f_i={\rm Id}_{\mathbb R}$ for each $i$, then every open interval in $\mathbb R$ is wandering by Definition 2.2. So,  we may as well suppose
that
\begin{equation}
\mathbb R\setminus {\rm Fix}(N_{k-1})\not=\emptyset.
\end{equation}

{\bf Claim A.} There exists a sequence of open intervals $J_i$, $0< i\leq k$, such that
$J_1\supset ... \supset J_k$, $J_i\subset \mathbb R\setminus {\rm Fix}(f_{i})$ for $i<k$, $f_{i-1}(J_i)\cap J_i=\emptyset$ and $f_{j}(J_i)=J_i$ for $j\geq i.$

We prove this claim by induction. Take a maximal open interval $J_{k-1}$ in $\mathbb R\setminus {\rm Fix}(N_{k-1})$ by $(4.2)$, and take an open interval $J_k\subset J_{k-1}$ such that $f_{k-1}(J_k)\cap J_k=\emptyset$.
Assume that, for some $l\geq 0$, we have obtained open intervals $J_{l+1}\supset... \supset J_k$ such that
$J_i\subset \mathbb R\setminus {\rm Fix}(f_{i})$ for each $i: k>i>l$,
$f_{i-1}(J_i)\cap J_i=\emptyset\ {\mbox {for}}\ i>l,\ {\mbox {and}}\ f_{j}(J_i)=J_i\ {\mbox {for}} \ j\geq i.$ Let $J_i=(\alpha_i, \beta_i)$, $k\geq i\geq l+1$.
Let $J_l=(\alpha_l, \beta_l)$ be the maximal open interval of $\mathbb R\setminus {\rm Fix}(f_{l})$, which contains $J_{l+1}$.
Since either $\alpha_l=\lim\limits_{m\to+\infty}f^m(\alpha_{l+1})$ or $\alpha_l=\lim\limits_{m\to-\infty}f^m(\alpha_{l+1})$,
and  $\alpha_{l+1}\in {\rm Fix}(N_{l+1})$, we have $\alpha_l\in {\rm Fix}(N_l)$. Similarly, $\beta_{l}\in {\rm Fix}(N_{l})$. So, $f_{j}(J_l)=J_l$ for $j\geq l.$ In addition, we have $f_{l-1}(J_{l})\cap J_{l}=\emptyset.$ Otherwise, $f_{l-1}(J_{l})=J_{l}$ since $f_{l-1}({\rm Fix}(N_{l}))={\rm Fix}(N_{l})$ by $(4.1)$. Without loss of generality, we suppose $f_{l}(x)>x$ for every $x\in J_{l}$.
Let $w=\alpha_{l+1}$.
Then $w\in {\rm Fix}(N_{l+1})$, and $\alpha_{l}=\lim\limits_{m\to+\infty}f_{l-1}f_{l}^{-m}(w)=\lim\limits_{m\to+\infty}f_{l}^{m}(f_{l-1}(w))=\beta_{l}$ by $(4.1)$, which is a contradiction.
Thus we complete the proof of Claim A.

{\bf Claim B.} $J_k$ is a wandering interval of $\phi$. In fact, for any $g\in G$, $\phi(g)$ can be expressed as
$\phi(g)=f_0^{n_0}f_1^{n_1}... f_{k-1}^{n_{k-1}}$ for some integers $n_0, n_1, ..., n_{k-1}$.
If $\phi(g)(J_k)\cap J_k\not=\emptyset$, then it follows from Claim A that $n_0=n_1=...=n_{k-1}=0$. Thus
$\phi(g)={\rm Id}_{\mathbb R}$.

From Claim B, we get that $G$ is of wandering type.
\end{proof}

\begin{lem}
Let $H$ be a normal subgroup of $G$ such that $G/H$ is infinite cyclic. Then
$G$ is of transitive type provided that $H$ is of transitive type.
\end{lem}

\begin{proof}
Suppose $G/H=\langle aH \rangle$ for some $a\in G\setminus H$. Let $a$ act on the line by the unit translation $\phi(a):x\mapsto x+1$. By the assumption, $H$
has an orientation-preserving topologically transitive action on $(0, 1)$, which extends to an action $\phi$ on the interval
$[0, 1]$ by fixing the endpoints. Then extend this $H$ action to $G$ action on the line by setting, for each $j\in\mathbb Z$,
\begin{equation}
\phi(b)(x)=\phi(a^{-j}ba^j)(x-j)+j
\end{equation}
for all $x\in[j,j+1]$ and $b\in H$. Then define $\phi$ on $G$ by setting $\phi(a^ib)=(\phi(a))^i\phi(b)$ for all $i\in\mathbb Z$ and $b\in H$.
It is direct to check that $\phi(g_1g_2)=\phi(g_1)\phi(g_2)$ for arbitrary two elements
$g_1=a^{l_1}b_1$ and $g_2=a^{l_2}b_2$ in $G$ where $b_1, b_2\in H$.
In fact, for all $x\in [i, i+1]$, we have
$$\begin{array}{rl}
\phi(g_1g_2)(x)&=\phi(a^{l_1}b_1a^{l_2}b_2)(x)\\

               &=\phi(a^{l_1+l_2}a^{-l_2}b_1a^{l_2}b_2)(x)\\

               &=\phi(a^{l_1+l_2})\phi(a^{-l_2}b_1a^{l_2}b_2)(x)\\

               &=\phi(a^{-i}a^{-l_2}b_1a^{l_2}b_2a^i)(x-i)+i+l_1+l_2,
\end{array}
$$
while
$$\begin{array}{rl}
\phi(g_1)\phi(g_2)(x)&=\phi(a^{l_1}b_1)\phi(a^{l_2}b_2)(x)\\

               &=\phi(a^{l_1}b_1)(\phi(a^{-i}b_2a^i)(x-i)+i+l_2)\\

               &=\phi(a^{-i-l_2}b_1a^{i+l_2})\phi(a^{-i}b_2a^{i})(x-i)+i+l_1+l_2\\

               &=\phi(a^{-i}a^{-l_2}b_1a^{l_2}b_2a^i)(x-i)+i+l_1+l_2,
\end{array}
$$
as required.
Thus
$\phi$ is an orientation-preserving action of $G$ on the line and the topological transitivity of $\phi$ is clear.
\end{proof}

\begin{prop}
Suppose $G=N_0\rhd N_1\rhd...\rhd N_k=\{e\}$ is super-poly-infinite-cyclic with $k\geq 2$ and has the name $(n_0, n_1, ..., n_{k-2})$ with some $n_i=1$.
Then $G$ is of transitive type.
\end{prop}

\begin{proof}
If $n_{k-2}=1$, then $N_{k-2}$ is isomorphic to $\mathbb Z^2$, which is of transitive type by Example 1.2. By repeated applications of Lemma 4.2, we get that
$G$ is of transitive type. If $n_i=1$ for some $i<k-2$, then $G/N_{i+2}$ is of transitive type by the previous argument. Since $G/N_{i+2}$ is a factor of $G$,
$G$ is of transitive type.
 \end{proof}

\begin{prop}
Suppose $G$ is poly-infinite-cyclic. If $G$ is not super-poly-cyclic, then $G$ is of transitive type.
\end{prop}

\begin{proof}
By Proposition 6.4 in the appendix, there is a decreasing sequence of normal subgroups of $G$: $G=N_0\rhd N_1\rhd...\rhd N_k=\{e\}$ for some $k>0$, such that
each $N_i/N_{i+1}$ is isomorphic to ${\mathbb Z}^{d_i}$ for some $d_i\geq 1$. If
every $N_i/N_{i+1}$ is isomorphic to ${\mathbb Z}$, then $G$ is super-poly-cyclic, which contradicts the assumption. So, there is some
$i'$ such that $N_{i'}/N_{i'+1}$ is isomorphic to ${\mathbb Z}^{d_{i'}}$ with $d_{i'}\geq 2$. Let ${\tilde G}=G/N_{i'+1}$.
Since $\tilde N_{i'}\equiv N_{i'}/N_{i'+1}$ is of transitive type by Example 1.2, ${\tilde G}$ is of transitive type by repeated applications of Lemma 4.2
(note that ${\tilde G}/\tilde N_{i'}$ is poly-infinite-cyclic). Then $G$ is of transitive type since $\tilde G$ is a factor of $G$.
\end{proof}

\begin{prop}
The Baumslag-Solitar group $B(1, n)$ is of transitive type if and only if $n\not=0$ and $n\not=-1$.
\end{prop}

\begin{proof} $(\Rightarrow)$ Suppose $B(1, n)=\langle a, b:  ba=a^nb \rangle$. If $n=0$, then $B(1, n)$ is an
infinite cyclic group; in this case, every orientation-preserving action of $B(1, n)$ on $\mathbb R$ has a wandering interval.
Supppose $n=-1$. Let $\phi: B(1, -1)\rightarrow {\rm Homeo}_+(\mathbb R)$ be any orientation-preserving action of $B(1, -1)$ on $\mathbb R$. Let $g=\phi(a)$ and $f=\phi(b)$; then $fg=g^{-1}f$. We wish to show that $\phi$ has a wandering interval.
If ${\rm Fix}(g)=\mathbb R$, then $\phi$ factors through a
cyclic group action on $\mathbb R$, which ensures the existence of wandering intervals. So, we may suppose that
${\rm Fix}(g)\not=\mathbb R$.

{\bf Claim A.} $f({\rm Fix}(g))={\rm Fix}(g)$. In fact, let $x\in {\rm Fix}(g)$. Then $g^{-1}f(x)=fg(x)=f(x)$.
So, $f(x)\in {\rm Fix}(g^{-1})={\rm Fix}(g)$. On the other hand, $g^{-1}f^{-1}(x)=f^{-1}g(x)=f^{-1}(x)$, which
means $f^{-1}(x)\in {\rm Fix}(g^{-1})={\rm Fix}(g)$.

From Claim A, we see that $f$ permutes the maximal open intervals in $\mathbb R\setminus {\rm Fix}(g)$.
Fix a maximal open interval $(u, v)$ in $\mathbb R\setminus {\rm Fix}(g)$ ($u$ may be $-\infty$, and $v$ may be $+\infty$).

{\bf Claim B.} $f((u, v))\cap (u, v)=\emptyset.$ Otherwise, $f((u, v))=(u, v)$. Without loss of generality, we suppose $g(x)>x$ for every $x\in (u, v)$.
Fix any point $w\in (u, v)$, then $v=\lim\limits_{i\to+\infty}fg^{i}(w)=\lim\limits_{i\to+\infty}g^{-i}(f(w))=u$ by the relation $fg=g^{-1}f$.
This is a contradiction.

From Claim A and Claim B, we immediately get

{\bf Claim C.} Suppose $f^{m_1}g^{n_1}f^{m_2}g^{n_2}...f^{m_l}g^{m_l}((u, v))=(u,v)$ for some integers $l$, $m_i$, and $n_i$ $(1\leq i\leq l)$. Then
$m_1+m_2+...+m_l=0$.

Take an open interval $J\subset (u, v)$ such that $J\cap g(J)=\emptyset$.

{\bf Claim D.} If $h(J)\cap J\not=\emptyset$ for some $h\in \phi(B(1, -1))$, then $h={\rm Id}_{\mathbb R}$. In fact, suppose $h=f^{m_1}g^{n_1}f^{m_2}g^{n_2}...f^{m_l}g^{m_l}$
such that $|m_1|+|n_1|+|m_2|+|n_2|+...+|m_l|+|n_l|$ attains minimum among all expressions of $h$ by $f$ and $g$. This implies that
all $m_i$ with $m_i\not=0$ have the same signs by the relation $f^{-1}gf=g^{-1}$. However, this forces all $m_i=0$ by Claim C (noting that $h((u, v))=(u, v)$).
Then $h={\rm Id}_{\mathbb R}$, since $J\cap g(J)=\emptyset$.

It follows from Claim D that $J$ is a wandering interval for $\phi$.

$(\Leftarrow)$ Since $B(1, 1)$ is isomorphic to $\mathbb Z^2$, it is of transitive type by Example 1.2. If $n>1$, then $B(1, n)$ is of transitive type
by Example 1.3 (note that $T$ and $S$ in Example 1.3 satisfy the relation $ST=T^nS$). In Example 1.4, we see that $f$ and $g$ satisfy
the relation $fg=g^{-k}f$ with $k\geq 2$, which implies that $B(1, n)$ is of transitive type when $n\leq -2$.

\end{proof}

From example 1.1, we see that the free nonabelian group $\mathbb Z*\mathbb Z$ has a topologically transitive action on $\mathbb R$
by orientation-preserving homeomorphisms, since $\mathbb Z^2$ is a factor of $\mathbb Z*\mathbb Z$. In fact, we can further require the action to be faithful as the following lemma shows.

\begin{prop}
The free nonabelian group $\mathbb Z*\mathbb Z$ has a faithful topologically transitive action on $\mathbb R$
by orientation-preserving homeomorphisms.
\end{prop}

\begin{proof}
Let $f, g\in {\rm Homeo}_+(\mathbb R)$ be defined by $f(x)=x+1$ and $g(x)=x^3$ for every $x\in \mathbb R$.  Then for
any nonempty open intervals $U$ and $V$ in $\mathbb R$, we have ${\rm diam}(g^n(U))>1$ for some integer $n$, and then
there is some integer $m$ such that $f^m(g^n(U))\cap V\not=\emptyset$. Thus the action by the group $H$ generated by $f$ and $g$
is topologically transitive. By the main result in \cite{Wh}, we see that $H$ is isomorphic to $\mathbb Z*\mathbb Z$.
\end{proof}

The following theorem is due to D. Witte-Morris (see \cite{Wi2}).
\begin{thm}[Witte-Morris]
The group $SL(2, \mathbb Z)$ and all finite index subgroups of $SL(n,\mathbb Z)$ with $n\geq 3$ are non-orderable.
\end{thm}

\begin{prop}
The group $SL(2, \mathbb Z)$ and all finite index subgroups of $SL(n,\mathbb Z)$ with $n\geq 3$ are of wandering type.
\end{prop}

\begin{proof}
Since $SL(2, \mathbb Z)$ is generated by elements with finite orders, any orientation preserving action of $SL(2, \mathbb Z)$
on $\mathbb R$ must be trivial. So,  $SL(2, \mathbb Z)$ is of wandering type. Suppose $n\geq 3$ and $H$ is a subgroup of $SL(n,\mathbb Z)$
with finite index. Assume $H$ is of transitive type and let $\phi:H\rightarrow {\rm Homeo}_+(\mathbb R)$ be a topologically transitive action
of $H$ on $\mathbb R$. By Selberg's Lemma (see Theorem 6.6 in the appendix), there is a torsion free normal subgroup $F$ of $H$ which has finite index in $H$. It follows
from Theorem 6.5 (see the appendix) and the topological transitivity of $\phi$ that ${\rm Ker}(\phi)$ is  finite,
which implies that ${\rm Ker}(\phi)\cap F$ is trivial. So the restriction $\phi|_F:F\rightarrow {\rm Homeo}_+(\mathbb R)$ is injective.
Thus $F$ is orderable, which contradicts Theorem 4.7.
\end{proof}

\begin{rem}
Since the free non-abelian group $\mathbb Z*\mathbb Z$ is a finite index subgroup of $SL(2, \mathbb Z)$ and $\mathbb Z*\mathbb Z$
is of transitive type, Proposition 4.8 does not hold for finite index subgroups of $SL(2, \mathbb Z)$.
\end{rem}

Then Theorem 2.3 follows from all the propositions in this section.

\section{Indicability}

To prove Theorem 2.4, we need several well-known results  about group actions on $\mathbb R$.
The following lemma can be shown by the dynamical realization method (see \cite[Theorem 2.2.19]{Na1} and its remark).

\begin{lem}
Every countable nontrivial orderable group has a faithful orientation-preserving action on the line $\mathbb R$ without fixed points.
\end{lem}

The following lemma is the combination of \cite[Proposition 2.1.12]{Na1} and the remarks after it (see also \cite{Ma}).

\begin{lem}
If $G$ is a finitely generated group acting on the line $\mathbb R$ by orientation-preserving homeomorphisms, then $G$ admits
a nonempty minimal closed subset $\Lambda$ of $\mathbb R$, and $\Lambda$ has four possibilities:

{\noindent\bf(a)} $\Lambda$ is a point (in this case, $\Lambda$ is a fixed point of $G$);

{\noindent\bf(b)} $\Lambda$ is an infinite sequence $(a_n)_{n\in\mathbb Z}$ satisfying $a_n<a_{n+1}$ for all $n$ and without accumulation points in $\mathbb R$;

{\noindent\bf(c)} $\Lambda$ is locally a Cantor set;

{\noindent\bf(d)} $\Lambda=\mathbb R$.

\end{lem}

{\noindent\it Proof of Proposition 2.5.} Suppose $G$ is a finitely generated nontrivial orderable group of wandering type. From Lemma 5.1, we can fix a
faithful action $\phi:G\rightarrow {\rm Homeo}_+(\mathbb R)$ without fixed points. By Lemma 5.2, there is a minimal set $\Lambda\subset\mathbb R$. Since
$\phi$ has no fixed points, $\Lambda$ cannot be a single point. By Definition 2.2 and Lemma 3.1, we see that $\Lambda$ is countable, which together with
Lemma 5.2 implies that $\Lambda$ is an infinite sequence $(a_n)_{n\in\mathbb Z}$ satisfying $a_n<a_{n+1}$ for all $n$ and without accumulation points in $\mathbb R$. Set $H=\{g\in G: g((a_0, a_1))=(a_0, a_1)\}$ and fix an $f\in G$ with $f(a_0)=a_1$. By the structure of $\Lambda$,
we have $H=\{g\in G: g(a_n)=a_n\ \mbox{for all}\ n\}$ and $f(a_n)=a_{n+1}$ for all $n$. Thus $H$ is normal in $G$, and $G/H=\{f^nH, n\in\mathbb Z\}$ which
is an infinite cyclic group. This completes the proof. \hfill{$\Box$}

\section{Appendix}

In this section, we first supply some basic results about poly-cyclic groups, which have been used in the precious sections.
One may consult \cite{Seg} for more details.

\begin{prop}
Let $G$ be a poly-cyclic group and let $H$ be a subgroup of $G$. Then $H$ is poly-cyclic.
\end{prop}

\begin{prop}[\cite{Rh}]
Let $G$ be a poly-cyclic group. Then $G$ is poly-infinite-cyclic if and only if $G$ is orderable.
\end{prop}

Since every subgroup of an orderable group is orderable, we immediately get the the following corollary by Proposition 6.1 and Proposition 6.2.

\begin{cor}
Let $G$ be a poly-infinite-cyclic group. Then every non-trivial subgroup of $G$ is poly-infinite-cyclic.
\end{cor}

Suppose $G$ is a poly-cyclic group.  Then $G=N_0\rhd N_1\rhd...\rhd N_k=\{e\}$ with $N_i/N_{i+1}$ cyclic for each $i$.
The cyclic groups $N_i/N_{i+1}$ are called the {\it cyclic factors}. The {\it Hirsch number} of $G$ is the number of infinite cyclic factors
among these $N_i/N_{i+1}$.  It is an invariant of polycyclic groups.

\begin{prop}
Let $G$ be a poly-infinite-cyclic group. Then there is a decreasing sequence of normal subgroups of $G$: $G=G_0\rhd G_1\rhd...\rhd G_l=\{e\}$ for some $l>0$, such that $G_i/G_{i+1}$ is a free abelian group of finite rank for each $i$.
\end{prop}

\begin{proof}
Suppose $G=N_0\rhd N_1\rhd...\rhd N_k=\{e\}$ for some $k>0$, where $N_{i+1}$ is normal in $N_i$ and $N_i/N_{i+1}$ is infinitely cyclic for each $i\geq 0$.
Since $G/N_1$ is abelian, the commutator group $[G, G]\subset N_1$. So, $G/[G, G]$ is isomorphic to $\mathbb Z^d\times F$ for some $d>1$, where $F$ is a finite
abelian group. Let $\pi:G\rightarrow G/[G, G]$ be the quotient homomorphism and let $G_1=\pi^{-1}(F)$. Then $G_1$ is a character subgroup of $G$ and $G/G_1$
is isomorphic to $\mathbb Z^d$. Similarly, we can get a character subgroup $G_2$ of $G_1$ (which is also a character subgroup
of $G$) such that $G_1/G_2$ is a free abelian group of finite rank, since $G_1$ is still poly-infinite-cyclic by  Proposition 6.3. Going on in this way,
we get a sequence of character subgroups of $G: G=G_0\rhd G_1\rhd... \rhd G_i\rhd...$ such that $G_i/G_{i+1}$ is a free abelian group of finite rank for each $i$. Since the Hirsch number of $G$ is finite,  there exists a positive integer $l$ such that $G_l=\{e\}$. Thus we complete the proof.
\end{proof}

The following theorem is due to Margulis and Kazhdan (see e.g. \cite[Theorem 8.1.2]{Zi}).

\begin{thm}
Let $\Gamma$ be a higher rank lattice and let $H$ be a normal subgroup of $\Gamma$. Then either $H$ is finite or $\Gamma/H$ is finite.
\end{thm}

The following theorem is known as Selberg's Lemma (see \cite{Sel}).
\begin{thm}
Let $G$ be a finitely generated subgroup of $GL(n, \mathbb C)$. Then $G$ contains a torsion free normal subgroup $H$ with finite index in $G$.
\end{thm}

\subsection*{Acknowledgements}
The work is supported by NSFC (No. 11771318, No. 11790274). We would like to thank Prof. Binyong Sun for providing us the idea of the proof of Proposition 6.4.

\end{document}